\documentclass[a4paper,%
			11pt
			]{article}

\usepackage[T1]{fontenc}
\usepackage[utf8]{inputenc}
\usepackage[british]{babel}
\usepackage{libertine}

\usepackage{amsmath}
	\numberwithin{equation}{section}
\usepackage{amssymb}
\usepackage{amsthm}
	\theoremstyle{plain}
		\newtheorem{thm}{Theorem}[section]
		\newtheorem{prop}[thm]{Proposition}
		\newtheorem{conv}[thm]{Convention}
	\theoremstyle{definition}
		\newtheorem{defn}[thm]{Definition}

		\newtheorem{example}[thm]{Example}
	\theoremstyle{remark}
		\newtheorem{rem}[thm]{Remark}
\usepackage{mathtools} 
\usepackage{mathrsfs} 
\usepackage{eucal} 
\usepackage{braket} 

\usepackage{microtype}
\usepackage{mparhack}

\usepackage{enumitem}
\usepackage{varioref}		

\usepackage{graphicx}
\usepackage{subfig}
\usepackage{booktabs}	
\usepackage{siunitx}		
\usepackage{caption}	
\usepackage{listings}
\usepackage{algorithm}
\usepackage{algpseudocode}

\usepackage[svgnames]{xcolor}

\usepackage{tikz}
	\usetikzlibrary{tikzmark,calc}
		\newcommand\DrawBox[3][]{%
			\begin{tikzpicture}[remember picture,overlay]
			\draw[overlay,fill=gray!30,#1]
				([xshift=-3.7em,yshift=2.1ex]{pic cs:#2})
				rectangle
				([xshift=2pt,yshift=-0.7ex]pic cs:#3);
			\end{tikzpicture}%
		}
	\algnewcommand\algorithmicinput{\textbf{Input:}}
	\algnewcommand\INPUT{\item[\algorithmicinput]}
	\algnewcommand\algorithmicoutput{\textbf{Output:}}
	\algnewcommand\OUTPUT{\item[\algorithmicoutput]}
	\algdef{SE}[FUNCTION]{Function}{EndFunction}[2]{\algorithmicfunction\ \textproc{#1}\ifthenelse{\equal{#2}{}}{}{(#2)}}{\algorithmicend\ \algorithmicfunction}

\usepackage[colorlinks=true, linkcolor=RoyalBlue, citecolor=ForestGreen, urlcolor=Crimson]{hyperref}

\DeclareFixedFont{\ttb}{T1}{txtt}{bx}{n}{12} 
\DeclareFixedFont{\ttm}{T1}{txtt}{m}{n}{12}  

\definecolor{deepblue}{rgb}{0,0,0.5}
\definecolor{deepred}{rgb}{0.6,0,0}
\definecolor{deepgreen}{rgb}{0,0.5,0}

\newcommand\pythonstyle{\lstset{
language=Python,
basicstyle=\ttm,
otherkeywords={self},             
keywordstyle=\ttb\color{deepblue},
emph={MyClass,__init__},          
emphstyle=\ttb\color{deepred},    
stringstyle=\color{deepgreen},
frame=tb,                         
showstringspaces=false, 
tabsize=2         %
}}

\lstnewenvironment{python}[1][]
{
\pythonstyle
\lstset{#1}
}
{}

\newcommand\pythoninline[1]{{\pythonstyle\lstinline!#1!}}

\newcommand{\N}{\mathbb{N}}

\newcommand{\R}{\mathbb{R}}
\newcommand{\GL}{\mathrm{GL}}
\newcommand{\Mat}{\mathrm{Mat}}

\newcommand{\dom}{\mathrm{dom}}
\newcommand{\abs}[1]{\left\lvert #1 \right\rvert} 
\newcommand{\trasp}[1]{#1^{\mathsf{T}}} 

\renewcommand{\leq}{\leqslant}
\renewcommand{\geq}{\geqslant}
\renewcommand{\=}{\coloneqq}	 

\newcommand{\ie}{i.\,e.~}

\title{A mathematical model for voice leading\\ and its complexity}
\author{Mattia G.~Bergomi, Riccardo D.~Jadanza, Alessandro Portaluri}

\begin{document}

\maketitle

\begin{abstract}
We give a formal definition of the musical concept of voice leading in mathematical terms, based on the idea of partial permutations of certain ordered multisets.
Then we associate a partial permutation matrix with a voice leading in a unique way and write an algorithm to easily transform any musical composition into a sequence of such matrices; we then generalise it in order
to include in the model also rhythmic independence and rests.
From that we extract a vector whose components return information about the movements of the voices in the piece and hence about the complexity of the voice leading.
We provide some examples by analysing three compositions, also visualising complexity as a point cloud for each piece.
Finally, we interpret the sequence of complexity vectors associated with each composition, thus considering the position of each observation with respect to time. The Dynamic Time Warping allows us to compute the distance between two pieces and to show that our approach distinguishes well the examples that we took into account, exhibiting a strong indication that the notion of complexity we propose is a good tool to identify and classify musical pieces.
\end{abstract}

\section{Introduction} \label{sec:intro}

In Music, the study of harmony and voice leading concerns the dynamical relation among chords and melodic lines.
Chords can be seen as vertical entities supporting a main melodic line, whereas compositional styles can be characterised by the interplay of the various voice lines in horizontal motion. Hence, voice leading can be
interpreted as the writing of several melodies interacting together in two ways: voices moving simultaneously affect the listener as chords, whilst the independence in terms of both rhythm and pitch creates a horizontal flow,
perceived as a superposition of different themes.

The concurrent motion of two voices is classically referred to as \emph{contrapuntal motion} and is traditionally divided into four classes, illustrated in Figure~\vref{fig:motions}.
In \emph{contrary motion} the voices move in opposite directions; this gives them contrast and independence \cite[Chapter 6]{aldwell2010harmony}, creating an interesting ``soundscape'' for the listener, as the Canadian
composer Raymond M.~Schafer calls it.
\emph{Parallel motion} occurs when the interval between the voices is kept constant along the movement; when applied to thirds, sixths and tenths it can be considered among the most powerful voice leading techniques.
However, in some cases it impedes the growth in independence of the voices: this is one reason why it is generally avoided when not even forbidden for unisons, octaves and fifths.
In \emph{oblique motion} actually only one voice is moving whilst the other remains at the same pitch; this one and \emph{similar motion}, where the two voices proceed in the same direction, convey surely less
independence than contrary motion, but undoubtably more than the parallel one.

For practical purposes, voices can be grouped in ranges: from the highest to the lowest, we have:
\begin{align*}
	\text{Soprano}&: \quad \text{from $C_4$ to $G_6$}, \\
	\text{Alto}&: \quad \text{from $G_3$ to $C_5$}, \\
	\text{Tenor}&: \quad \text{from $C_3$ to $G_4$}, \\
	\text{Bass}&: \quad \text{from $E_2$ to $C_4$}.
\end{align*}

\noindent
Different compositional styles are characterised by different types of motion: the five species of counterpoint represented in Figure~\ref{fig:species} arise from various combinations of rhythmic choices and of the amplitudes
of the intervals between two consecutive notes.
A particular movement of voices is the \emph{(voice) crossing}, that occurs when two voices exchange their relative positions --- for instance when the Soprano moves below the Alto.
This kind of dynamics is generally not desirable because it conveys a sense of discomfort, albeit it is considered less problematic when it involves inner voices (Alto and Tenor) for a few chords.

\begin{figure}[tb]
\centering
\includegraphics[width=\textwidth]{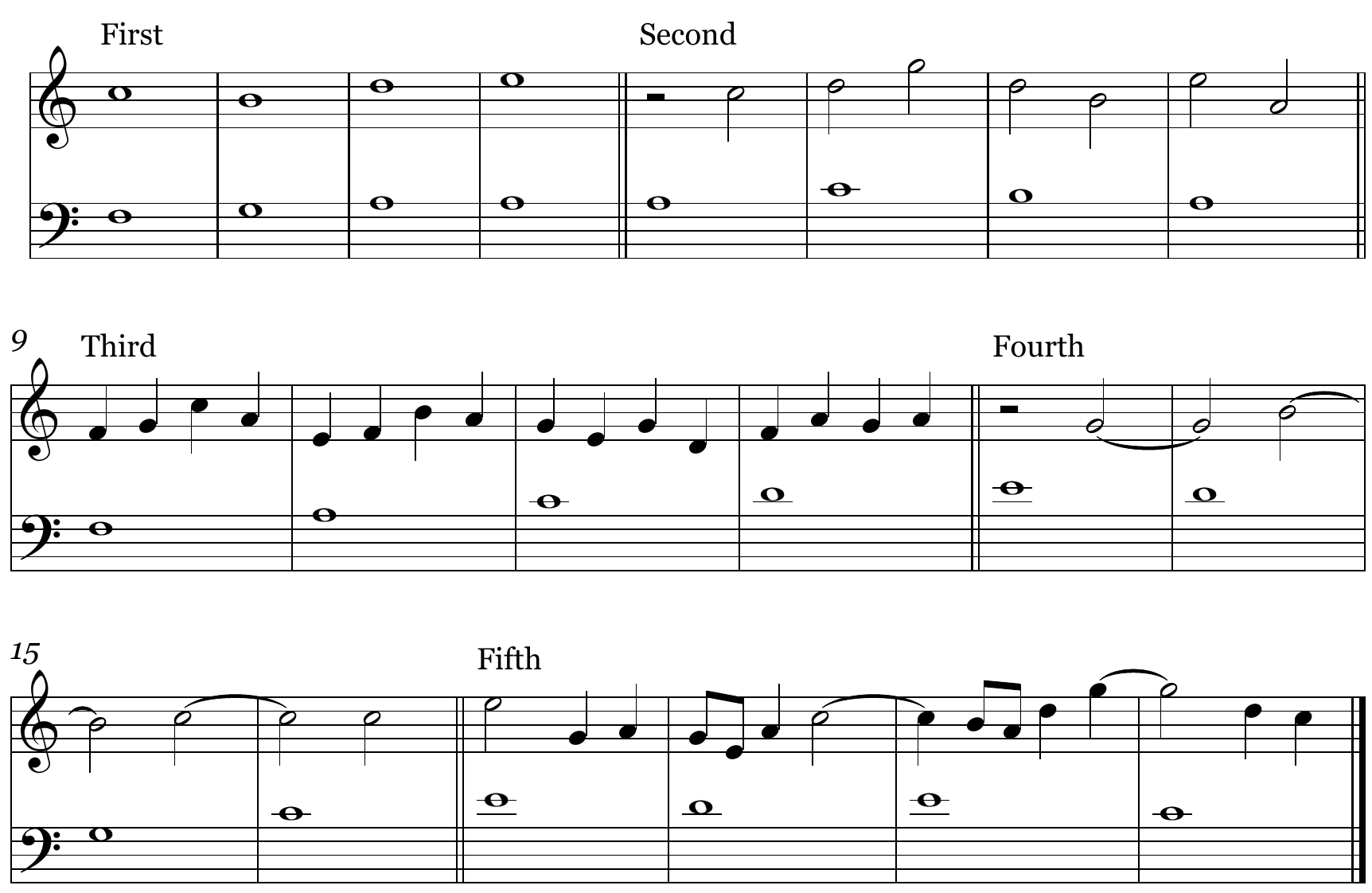}
\caption{Five different degrees of independence between voices, from the tied one-to-one in the first species of counterpoint to the complete independence of the fifth species.} \label{fig:species}
\end{figure}

Given a sequence of chords, an important question is how to transform them into a superposition of voices according to a certain contrapuntal style (see Figure~\ref{fig:chordmel}).
Our aim is to formalise the concept of voice leading in mathematical terms and to build a computationally efficient model for dealing with voice leading based on sparse matrices that encode the motions of the involved voices.
This representation allows us to define a notion of \emph{complexity} of the voice leading and to classify different contrapuntal styles, by representing an entire composition as a static point cloud or as a
multi-dimensional time series. This second interpretation offers also the possibility to compare two pieces by making use of the \emph{dynamic time warping}, a tool that measures the similarity between them in terms of
complexity.

\begin{figure}[tb]
\centering
\includegraphics[width=\textwidth]{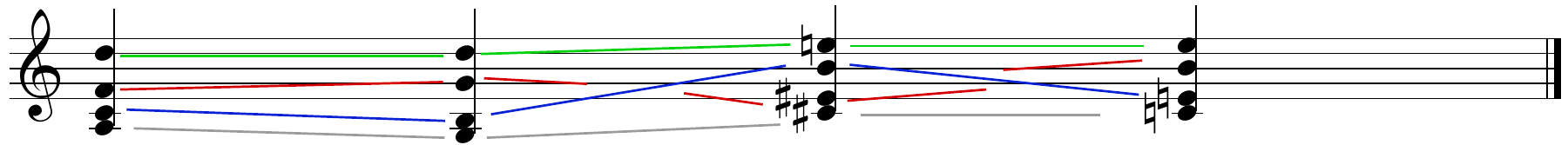}
\caption{From a sequence of chords to a superposition of melodies.} \label{fig:chordmel}
\end{figure}

\section{Voice leadings, multisets and partial permutations}

In general, it is possible to describe a melody as a finite sequence of ordered pairs $(p_i, p_{i+1})_{i \in I}$, where $I$ is a finite set of indices. In order to model the voice leading in a mathematical way it is necessary to
introduce first the concept of \emph{multiset}, a generalisation of the idea of set. (This approach was already considered by D.~Tymoczko in \cite{Tym06}.)
Roughly speaking, we can think of it as of a list where an object can appear more than once, whilst the elements of a set are necessarily unique.
More formally, a \emph{multiset} $M$ is a couple $(X, \mu)$ composed of an \emph{underlying set} $X$ and a map $\mu : X \to \N $, called the \emph{multiplicity} of $M$, such that for every $x \in X$ the value $\mu(x)$ is
the number of times that $x$ appears in $M$. We define the \emph{cardinality} $\abs{M}$ of $M$ to be the sum of the multiplicities of each element of its underlying set $X$.
Observe, however, that a multiset is in fact completely defined by its multiplicity function: it suffices to set $M \= \big( \dom(\mu), \mu \big)$.

If we interpret a set of $n$ singing voices (or parts played by $n$ instruments, or both) as a multiset of pitches of cardinality $n$, then a voice leading can be mathematically described as follows.

\begin{defn}
Let $M \= (X_M, \mu_M)$ and $L \= (X_L, \mu_L)$ be two multisets of pitches with same cardinality $n$ and arrange their elements into $n$-tuples $(x_1, \dotsc, x_n)$ and $(y_1, \dotsc, y_n)$ respectively.%
\footnote{These are in fact the images of two bijective maps $\psi_M : \{ 1, \dotsc, n \} \to M$ and $\psi_L : \{ 1, \dotsc, n \} \to L$.}
A \emph{voice leading} of $n$ voices between $M$ and $L$, denoted by $(x_1, \dotsc, x_n) \to (y_1, \dotsc, y_n)$, is the multiset
\[
	Z \= \big\{ (x_1, y_1), \dotsc, (x_n, y_n) \big\},
\]
whose underlying set is $X_Z \= X_M \times X_L$ and whose multiplicity function $\mu_Z$ is defined accordingly, by counting the occurrences of each ordered pair.
\end{defn}

\begin{rem}
Observe that the definition just given is not linked to the particular type of object (pitches): it is possible to describe voice leadings also between pitch classes, for instance.
\end{rem}

Note that it is also possible to describe a voice leading as a bijective map from the multiset $M$ to the multiset $L$, \ie as a partial permutation of the \emph{union multiset} 		
\[
	M \cup L \= (X_M \cup X_L, \mu_{M \cup L}),
\]
where
\[
	\mu_{M \cup L} \= \max \{ \mu_M \chi_M, \mu_L \chi_L\}
\]
and $\chi_M$ and $\chi_L$ are the characteristic functions of $X_M$ and $X_L$, respectively.%
\footnote{For a multiset $S$ we assume that $\mu_S(x) = 0$ if $x \notin X_S$. With this understanding, the function $\mu_{M \cup L}$ is defined on the whole of $X_M \cup X_L$.}
We recall that a \emph{partial permutation} of a finite multiset $S$ is a bijection between two subsets of $S$. In general, if $S$ has cardinality $n$ then this map can be represented as an $n$-tuple of symbols,
some of which are elements of $S$ and some others are indicated by a special symbol --- we use $\diamond$ --- to be interpreted as a ``hole'' or an ``empty character''.
However, since we are not dealing with subsets of a fixed multiset, we shall use the cycle notation to avoid ambiguity and confusion.

\begin{rem}
In order to be able to do computations with partial permutations, it is fundamental to fix an \emph{ordering} among the elements of the union multiset $M \cup L$.
We henceforth give $M \cup L$ the natural ordering $\leq$ of real numbers, being its elements pitches. Indeed, in classical music with equal temperament, one defines the pitch $p$ of a note as a function of the fundamental
frequency $\nu$ (measured in Hertz) associated with the sound; more precisely, as the map $p : (0, +\infty) \to \R$ given by
\[
	p(\nu) \= 69 + 12 \log_2\left( \frac{\nu}{440} \right).
\]
This can be done also in the case where the elements of the union multiset are pitch classes: the ordering is induced by the ordering of their representatives belonging to a same octave. However, in this paper we shall not
follow this practice and shall instead restrict to pitches only.
\end{rem}

\begin{example}
The voice leading 
\begin{equation} \label{vl:example}
	(G_2, G_3, B_3, D_4, F_4) \to (C_3, G_3, C_4, C_4, E_4)
\end{equation}
is described by the partial permutation of the ordered union multiset 
\[
	\left(G_2, C_3, G_3, B_3, C_4, C_4, D_4, E_4, F_4\right)
\]
defined by
\begin{equation} \label{eq:pperm}
	\begin{pmatrix}
		 G_2 & C_3 & G_3 & B_3 & C_4 & C_4 & D_4 & E_4 & F_4 \\
		 C_3 & \diamond & G_3 & C_4 & \diamond & \diamond & C_4 & \diamond & E_4
	\end{pmatrix}.
\end{equation}
\end{example}

\noindent
Thus, a voice leading between two multisets of $n$ voices can be seen as a partial permutation of a multiset whose cardinality is less than or equal to $2n$.

The next step is to associate a representation matrix with the partial permutation.
Let $V$ be an $n$-dimensional vector space over a field $\mathbb{F}$ and let $\mathcal{E} \= \{e_1, \dotsc, e_n\}$ be a basis for $V$. 
The symmetric group $S_n$ acts on $\mathcal{E}$ by permuting its elements: the corresponding map $S_n \times \mathcal{E} \to \mathcal{E}$ assigns $(\sigma, e_i) \mapsto e_{\sigma(i)}$
for every $i \in \{1, \dotsc, n\}$. We consider the well-known \emph{linear representation} $\rho : S_n \to \GL(n, \mathbb{F})$ of the group $S_n$ given by
\[
\rho(1\ i) \= \left( \begin{smallmatrix}
          		0 & & & & 1 & & & \\
			   & 1 & & & & & & \\
			   & & \ddots & & & & & \\
			   & & & 1 & & & & \\
			1 & & & & 0 & & & \\
			   & & & & & 1 & & \\
			   & & & & & & \ddots & \\
			   & & & & & & & 1
		\end{smallmatrix} \right),
\]
where the $1$'s in the first row and in the first column occupy the positions $1, i$ and $i, 1$ respectively.
The map $\rho$ sends each $2$-cycle of the form $(1\ i)$ to the corresponding permutation matrix that swaps the first element of the basis $\mathcal{E}$ for the $i$-th one.
Note that each row and each column of a permutation matrix contains exactly one $1$ and all its other entries are $0$.
Following this idea and \cite[Definition~3.2.5, p.~165]{HJ91}, we say that a matrix $P \in \Mat(m, \R)$ is a \emph{partial permutation matrix} if for any row and any column there is at most one non-zero element
(equal to $1$). When dealing with a voice leading $M \to L$, the dimension $m$ of the matrix $P$ is equal to the cardinality of the multiset $M \cup L$.

\begin{rem}
In general, the partial permutation matrix associated with a given voice leading is not unique. This is due to the fact that we are dealing with multisets: if $M \to L$ is a voice leading it is possible that some components of
$L$ have the same value, \ie that different voices are playing or singing the same note.
\end{rem}

\noindent
For this reason we introduce the following convention.

\begin{conv} \label{fact:convention}
Let $M \= (x_1, \dotsc, x_n) \to L \= (y_1, \dotsc, y_n)$ be a voice leading and suppose that more than one voice is associated with a same note of $L$. To this end, let $(x_{i_1}, \dotsc, x_{i_k})$ be the pitches of $M$
(with $i_1< \dotsb < i_k$) that are mapped to the pitches $(y_{j_1}, \dotsc, y_{j_k})$ of $L$, with $y_{j_1} = \dotsb = y_{j_k}$ and $j_1 < \dotsb < j_k$.
In order to uniquely associate a partial permutation matrix $P \= (a_{ij})$ with the above voice leading, we assign the value $1$ to the corresponding entries of $P$ by following the order of the indices,
that is by setting $a_{i_1 j_1} = 1, \dotsc, a_{i_k j_k} = 1$.
\end{conv}

\noindent
Thus, we shall henceforth speak of \emph{the} partial permutation matrix associated with a given voice leading.

\begin{example}
The partial permutation matrix associated with the cycle representation \eqref{eq:pperm} of voice leading~\eqref{vl:example} is
\[
\begin{pmatrix}
	0 & 1 & 0 & 0 & 0 & 0 & 0 & 0 & 0 \\
	0 & 0 & 0 & 0 & 0 & 0 & 0 & 0 & 0 \\
	0 & 0 & 1 & 0 & 0 & 0 & 0 & 0 & 0 \\
	0 & 0 & 0 & 0 & 1 & 0 & 0 & 0 & 0 \\
	0 & 0 & 0 & 0 & 0 & 0 & 0 & 0 & 0 \\
	0 & 0 & 0 & 0 & 0 & 0 & 0 & 0 & 0 \\
	0 & 0 & 0 & 0 & 0 & 1 & 0 & 0 & 0 \\
	0 & 0 & 0 & 0 & 0 & 0 & 0 & 0 & 0 \\
	0 & 0 & 0 & 0 & 0 & 0 & 0 & 1 & 0
\end{pmatrix}.
\]
\end{example}

\noindent
Therefore, if $M \to L$ is a voice leading, if both $M$ and $L$ are thought of as ordered tuples and if $P$ is its partial permutation matrix, we have that $PM = L$; in addition, the ``reversed'' voice leading $L \to M$ is
obviously described by the transpose $\trasp{P}$ of $P$: $\trasp{P}L = M$.

This representation has the advantage of providing objects that are much handier than a multiset of couples, speaking in computational terms.
Algorithm~\vref{alg:ppmat} presents the pseudocode for the computation of the partial permutation matrix of a voice leading. 

\begin{algorithm}
\caption{Computing the partial permutation matrix} \label{alg:ppmat}
\DrawBox{c}{d}
\begin{algorithmic}[1]
\INPUT\tikzmark{c} 
\Statex $M \to L$ \Comment{Source ($M$) and target ($L$) multisets describing the voice leading}
\OUTPUT
\Statex $P$ \Comment{Partial permutation matrix associated with the voice leading\;\;\,\,\,}
\tikzmark{d}
\medskip
\Statex Evaluate multiplicities of all $x \in M$ and all $y \in L$;
\Statex Generate the \emph{ordered} multiset $U := M \cup L$;
\Statex Initialise $P \in \Mat\big(\!\abs{U}, \R \big)$ by setting $P(i,j)=0$ for all $i,j$; 
\For{$i,j\in\{1,\dots, \abs{U}\}$}
\If{$U(i) \to U(j)$}
\State $P(i,j) = 1$
\EndIf
\EndFor
\end{algorithmic}
\end{algorithm}

\section{Relative motion of the voices and complexity of a voice leading}

We have seen in the previous section how the partial permutation matrix associated with a voice leading contains the information about the passage from one note to the next one for each voice.
Here we are going to illustrate that, in fact, the tool that we have built also encodes the direction of motion of the different voices, including the crossings.

On the one hand, in Music one distinguishes between three main behaviours (cf.~Figure~\ref{fig:motions}; we omit parallel motion because that is not involved in our analysis):
\begin{itemize}
\item \emph{Similar motion}, when the voices move in the same direction;
\item \emph{Contrary motion}, when the voices move in opposite directions;
\item \emph{Oblique motion}, when only one voice is moving.
\end{itemize}
\begin{figure}[tb]
\centering
\includegraphics[width=\textwidth]{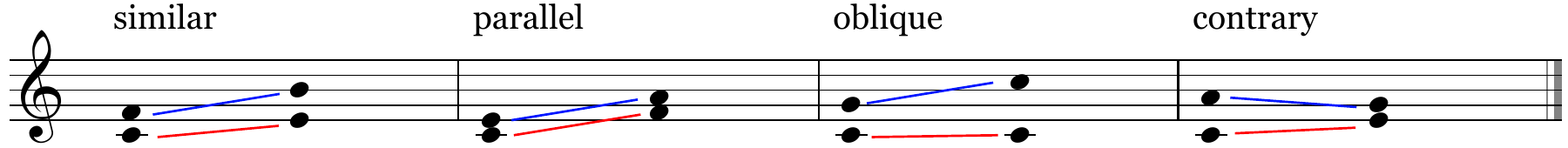}
\caption{Motion classes for two voices. \emph{Similar:} same direction but different intervals; \emph{parallel:} same direction and same intervals; \emph{oblique:} only one voice is moving;
		\emph{contrary:} opposite directions.}
\label{fig:motions}
\end{figure}
%
On the other hand, with reference to a partial permutation matrix $(a_{ij})$, it is possible to describe the motion of a voice by noting three conditions, which are immediate consequences of the ordering of the union multiset:
\begin{enumerate}[label=\arabic*)]
\item If there exists an element $a_{ij} = 1$ for $i < j$ then the $i$-th voice is moving ``upwards'';
\item If there exists an element $a_{ij} = 1$ for $i > j$ then the $i$-th voice is moving ``downwards'';
\item If there exists an element $a_{ii} = 1$ then the $i$-th voice is constant.
\end{enumerate}
The connection between the two worlds is the following:
\begin{itemize}
\item If either Condition~1) or Condition~2) is verified by two distinct elements then we have \emph{similar motion};
\item If both Condition~1) and Condition~2) hold for two distinct elements then we are facing \emph{contrary motion};
\item The case of \emph{oblique motion} involves Conditions~1) and 3) or Conditions~2) and 3), for at least two distinct elements.
\end{itemize}

As we mentioned in Section~\ref{sec:intro}, \emph{voice crossing} is a particular case of these motions where the voices swap their relative positions.
This phenomenon can be described in terms of multisets as follows.

\begin{defn} \label{def:crossing}
Let $(x_1, \dotsc, x_n) \to (y_1, \dotsc, y_n)$ be a voice leading ($n \in \N$). If there exist two pairs $(x_i, y_i)$ and $(x_j, y_j)$ such that
$x_i < x_j$ and $y_i > y_j$ or such that $x_i > x_j$ and $y_i < y_j$ then we say that a \emph{(voice) crossing} occurs between voice $i$ and voice $j$.
\end{defn}

\noindent
The partial permutation matrix retrieves even this information, as the following proposition shows.

\begin{prop} 
Consider a voice leading of $n$ voices and let $P \= (a_{ij})$ be its associated partial permutation matrix. Choose indices $i, j, k, l \in \{1, \dots, n\}$ such that $a_{ij} = 1$ and $a_{kl} = 1$. Then there is a crossing between these two voices if and only if one of the following conditions hold:
\begin{enumerate}[label=\roman*)]
\item $i < k$ and $j > l$;
\item $i > k$ and $j < l$.
\end{enumerate}
Furthermore, the total number of voices that cross the one represented by $a_{ij}$ is equal to the number of $1$'s in the submatrices $(a_{rs})$ and $(a_{tu})$ of $P$ determined by the following restrictions on the indices:
$r > i$, $s < j$ and $t < i$, $u > j$.
\end{prop}

\begin{proof}
In a partial permutation matrix the row index of a non-zero entry denotes the initial position of a certain voice in the ordered union multiset, whereas the column index of the same entry represents its final position after the
transition. It is then straightforward from Definition~\ref{def:crossing} that for a voice crossing to exist either condition i) or condition ii) must be verified.
Every entry $a_{kl}$ satisfying one of those conditions refers to a voice that crosses the one represented by $a_{ij}$, hence the number of crossings for $a_{ij}$ equals the amount of $1$'s in positions $(r, s)$ such that $r > i$
and $s < j$, summed to the number of $1$'s in positions $(t, u)$ such that $t < i$ and $u > j$.
\end{proof}

\begin{rem}
The fact that the number of crossings with a given voice equals the number of $1$'s in the submatrices determined by the entry corresponding to that voice (as explained in the previous proposition) holds true only because
we assumed Convention~\ref{fact:convention}. Indeed, if we did not make such an assumption, the submatrices could contain positive entries referring to voices ending in the same note but that do not produce crossings.
\end{rem}

From what we have shown thus far it emerges that it is possible to characterise a voice leading by counting the voices that are moving upwards, those that are moving downwards, those that remain constant
and the number of crossings. We summarise these features in a $4$-dimensional \emph{complexity vector} $c$ defined by
\begin{equation} \label{eq:complexvect}
	c \= \big(\#\text{upward voices},\ \#\text{downward voices},\ \#\text{constant voices},\ \#\text{crossings} \big),
\end{equation}
so that we are now able to classify and distinguish voice leadings by simply looking at these four aspects.

\begin{example}
\emph{Similar motion.} The voice leading $(C_1, E_1, G_1) \to (D_1, F_1, A_1)$ is represented by
\[
\begin{pmatrix}
	0 & 1 & 0 & 0 & 0 & 0\\
	0 & 0 & 0 & 0 & 0 & 0\\
	0 & 0 & 0 & 1 & 0 & 0\\
	0 & 0 & 0 & 0 & 0 & 0\\
	0 & 0 & 0 & 0 & 0 & 1\\
	0 & 0 & 0 & 0 & 0 & 0 
\end{pmatrix}
\]
and its complexity vector is $(3, 0, 0, 0)$.

\emph{Oblique motion.} The voice leading $(G_2 , G_2, C_3) \to (C_3,C_3,C_3)$ is associated with
\[
\begin{pmatrix}
	0 & 0 & 1 & 0 & 0\\
	0 & 0 & 0 & 1 & 0\\
	0 & 0 & 0 & 0 & 0\\
	0 & 0 & 0 & 0 & 0\\
	0 & 0 & 0 & 0 & 1
\end{pmatrix}
\]
and its complexity vector is $(2, 0, 1, 0)$.

\emph{Voice crossing.} The voice leading $ (C_1 , E_1, G_1)\to (G_1 , C_1 , E_1)$ is represented by
\[
\begin{pmatrix}
	0 & 0 & 1 \\
	1& 0 & 0  \\
	0 & 1 & 0 
\end{pmatrix} 
\]
and its complexity vector is $(1, 2, 0, 2)$.
\end{example}

%

By virtue of these tools it is straightforward to analyse an entire piece of music: it is enough to divide it into pairs of notes for each voice and apply the procedure described above for each passage.
The concatenation of all the consecutive passages results then in a sequence of partial permutation matrices, whence one can extract a sequence of complexity vectors. This last piece of information can be visualised as a
set of points in a $4$-dimensional space --- or rather as one or more of its $3$-dimensional projections (see Subsection~\ref{subsec:examples}). In fact, if one wants to represent the complexity of the whole composition as a
point cloud, one should take into account that different matrices can produce the same complexity vector; therefore we have a \emph{multiset} of points in $\R^4$ (with non-negative integer components).

\subsection{Complexity analysis of two \emph{Chartres Fragments}} \label{subsec:examples}

We are going to analyse two pieces that are parts of the \emph{Chartres Fragments}, an ensemble of compositions dating back to the Middle Ages: \emph{Angelus Domini} and \emph{Dicant nunc Judei};
both of them are counterpoints of the first species and involve only two voices.
The musical interest in these compositions consists in the introduction of a certain degree of independence between the voices and the use of a \emph{parsimonious voice leading}, \ie an attempt to make the passage from a
melodic state to the next as smooth as possible. Note how the independence of the voices is reflected by the presence of contrary motions and crossings, which can then be interpreted as a rough measure of this feature.
For a complete treatise on polyphony and a historical overview we refer the reader to \cite{taruskin2009music}.

In what follows, we represent the multiplicity of each complexity vector $c$ as a circle of centre $c \in \R^4$ and radius equal to the \emph{normalised multiplicity} $\mu(c)/n$ of $c$, where $\mu(c)$ is the number of
occurrences of $c$ in the analysed piece and $n$ is the total number of notes played or sung by each voice in the whole piece.

\paragraph{Angelus Domini.}

The fragment under examination is depicted in Figure~\ref{fig:indep}; here is the list of its first four voice leadings, as they are generated by the pseudocode described in Algorithm~\ref{alg:ppmat}:

\begin{figure}[tb]
\centering
\includegraphics[width=\textwidth]{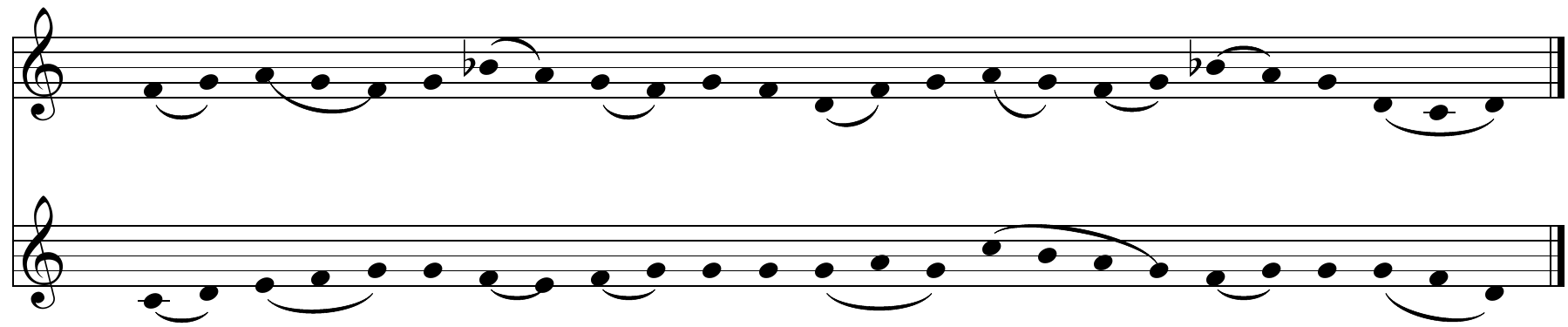} 
\caption{\emph{Angelus Domini}} \label{fig:indep}
\end{figure}

\begin{verbatim}
Voice Leading: ['F4', 'C4'] ['G4', 'D4']
[2, 0, 0, 0] - similar motion up
Voice Leading: ['G4', 'D4'] ['A4', 'E4']
[2, 0, 0, 0] - similar motion up
Voice Leading: ['A4', 'E4'] ['G4', 'F4']
[1, 1, 0, 0] - contrary motion
Voice Leading: ['G4', 'F4'] ['F4', 'G4']
[1, 1, 0, 1] - contrary motion - 1 crossing
\end{verbatim}

\noindent
Table~\vref{tab:complexvect} contains the the complexity vectors and their occurrences in the piece; the point cloud associated with this multiset is represented in Figure~\ref{fig:all}. Observe how the projection that neglects
the component of $c$ corresponding to the number of constant voices (Figure~\ref{fig:all2}) gives an immediate insight on the relevance of voice crossing in the piece.

\begin{figure}[tb]
\centering
\subfloat[Projection neglecting the \emph{crossing} component of the complexity vectors. \label{fig:all1}]{\includegraphics[width = 0.45\textwidth]{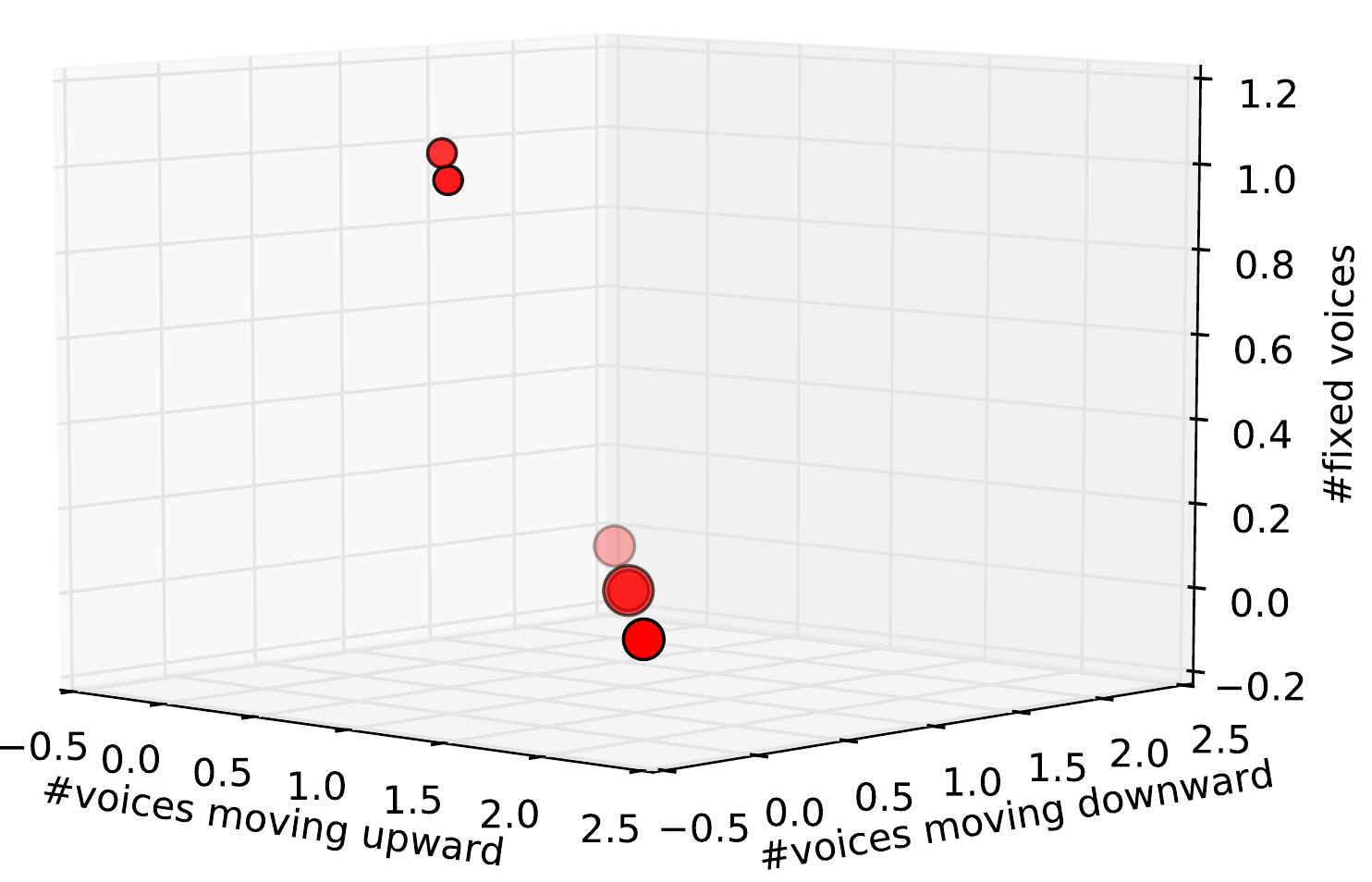}} \qquad
\subfloat[Projection neglecting the \emph{constant voices} component of the complexity vectors. \label{fig:all2}]{\includegraphics[width = 0.45\textwidth]{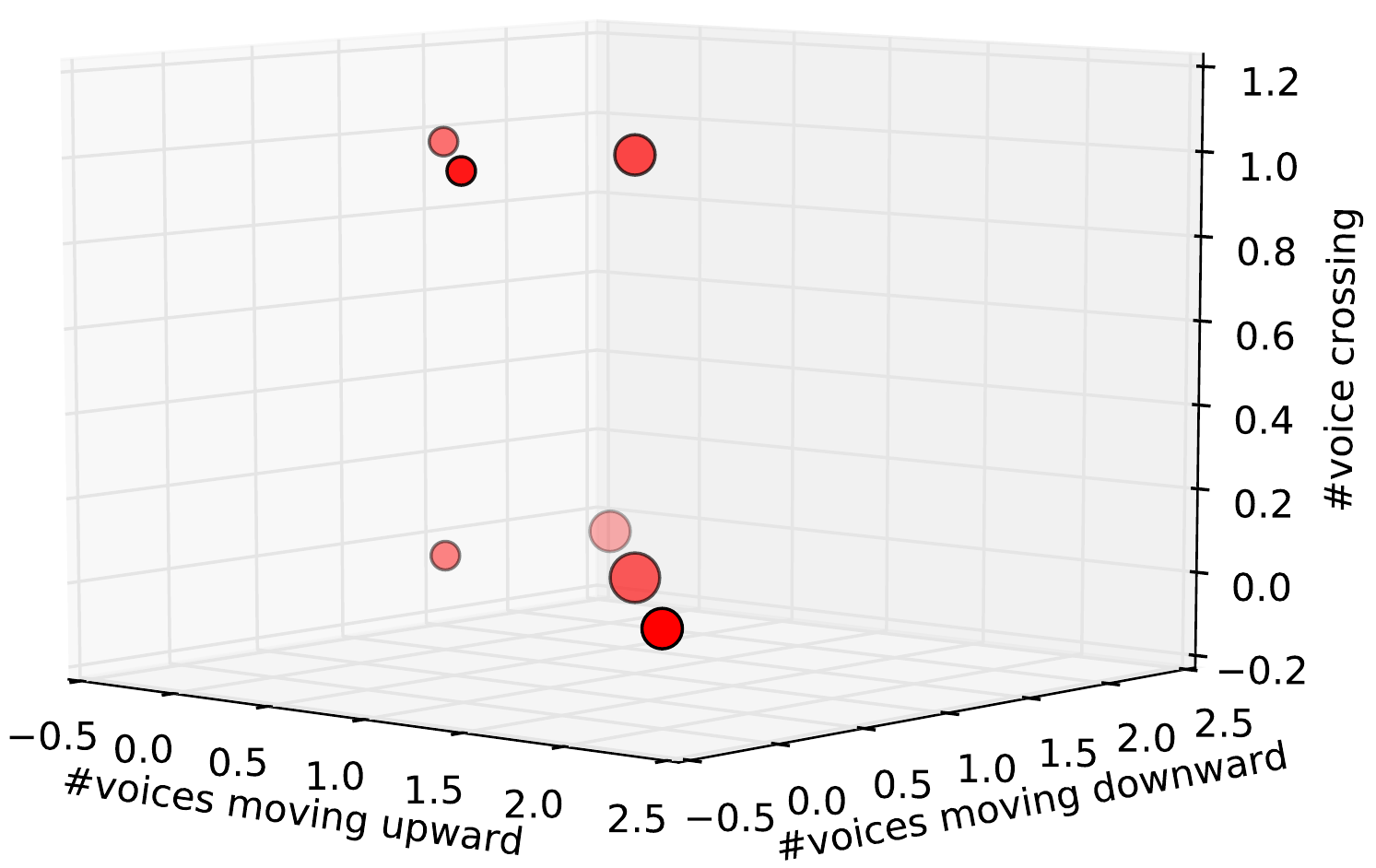}}
\caption{Three-dimensional projections of the complexity cloud of the paradigmatic voice leading \emph{Angelus Domini}.
		The radius of each circle represents the normalised multiplicity of the corresponding complexity vector.} \label{fig:all}
\end{figure}

\paragraph{Dicant nunc Judei.}

The first part of the output of Algorithm~\ref{alg:ppmat} produces the following analysis: 

\begin{verbatim}
Voice Leading: ['F4', 'C4'] ['G4', 'E4']
[2, 0, 0, 0] - similar motion up
Voice Leading: ['G4', 'E4'] ['F4', 'D4']
[0, 2, 0, 0] - similar motion down
Voice Leading: ['F4', 'D4'] ['E4', 'C4']
[0, 2, 0, 0] - similar motion down
Voice Leading: ['E4', 'C4'] ['D4', 'D4']
[1, 1, 0, 1] - contrary motion - 1 crossing
\end{verbatim}

The complexity vectors arising in the whole piece and their multiplicities are again collected in Table~\ref{tab:complexvect}; see Figure~\ref{fig:dic} instead for a visualisation of the point cloud describing the piece.
Note how the voice crossing is more massive than in the point cloud describing \emph{Angelus Domini}.
In addition, the point $(0, 0, 0)$ in Figure~\ref{fig:dic2} corresponds to the point $(0, 0, 2, 0) \in \R^4$, that represents trivial voice leadings where both parts do not vary. 

\begin{figure}[tb]
\centering
\includegraphics[width=\textwidth]{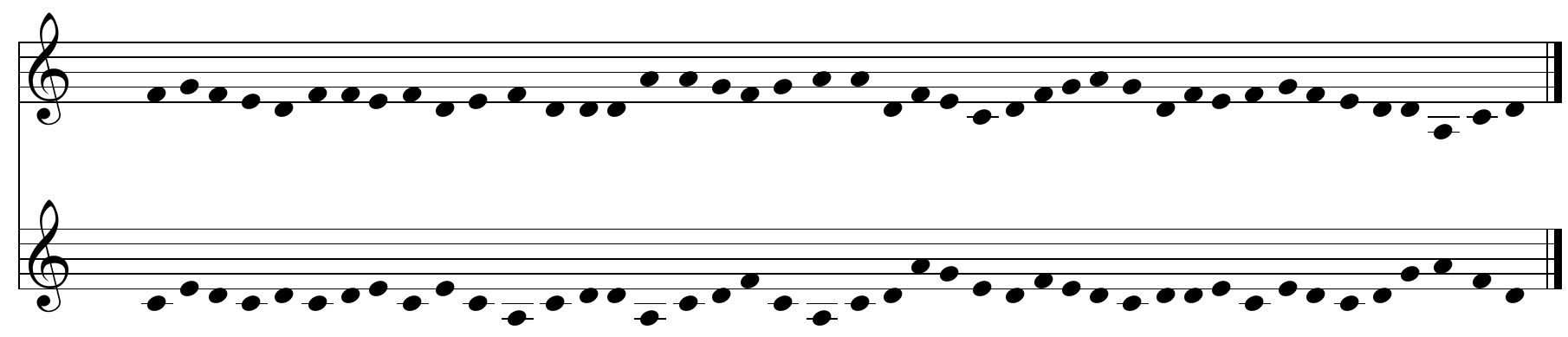} 
\caption{\emph{Dicant nunc Judei}, Chartres fragment.} \label{fig:dicant_sheet}
\end{figure}

\begin{figure}[tb]
\centering
\subfloat[Projection neglecting the \emph{crossing} component of the complexity vectors. \label{fig:dic1}]{\includegraphics[width = 0.45\textwidth]{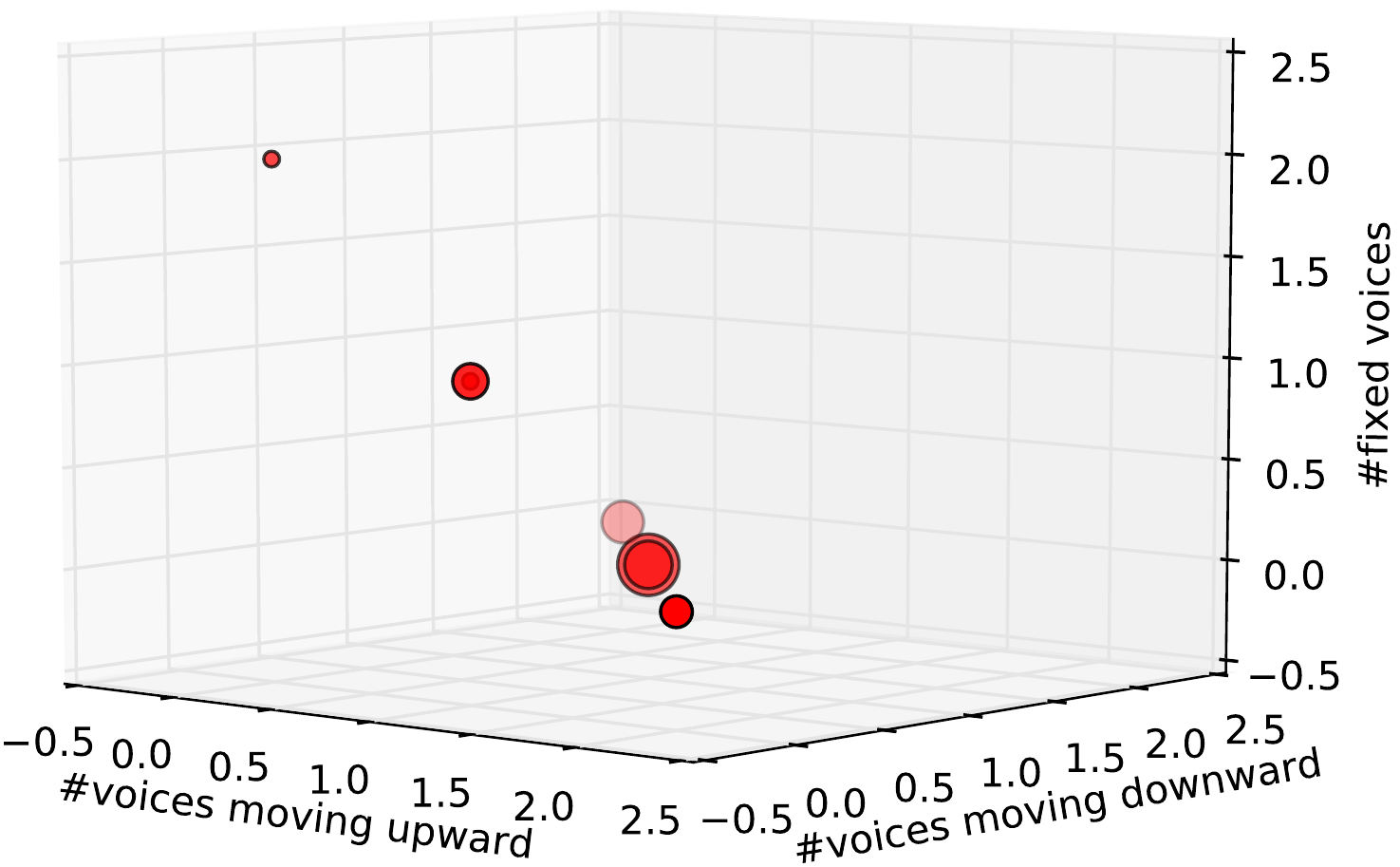}} \qquad
\subfloat[Projection neglecting the \emph{constant voices} component of the complexity vectors. \label{fig:dic2}]{\includegraphics[width = 0.45\textwidth]{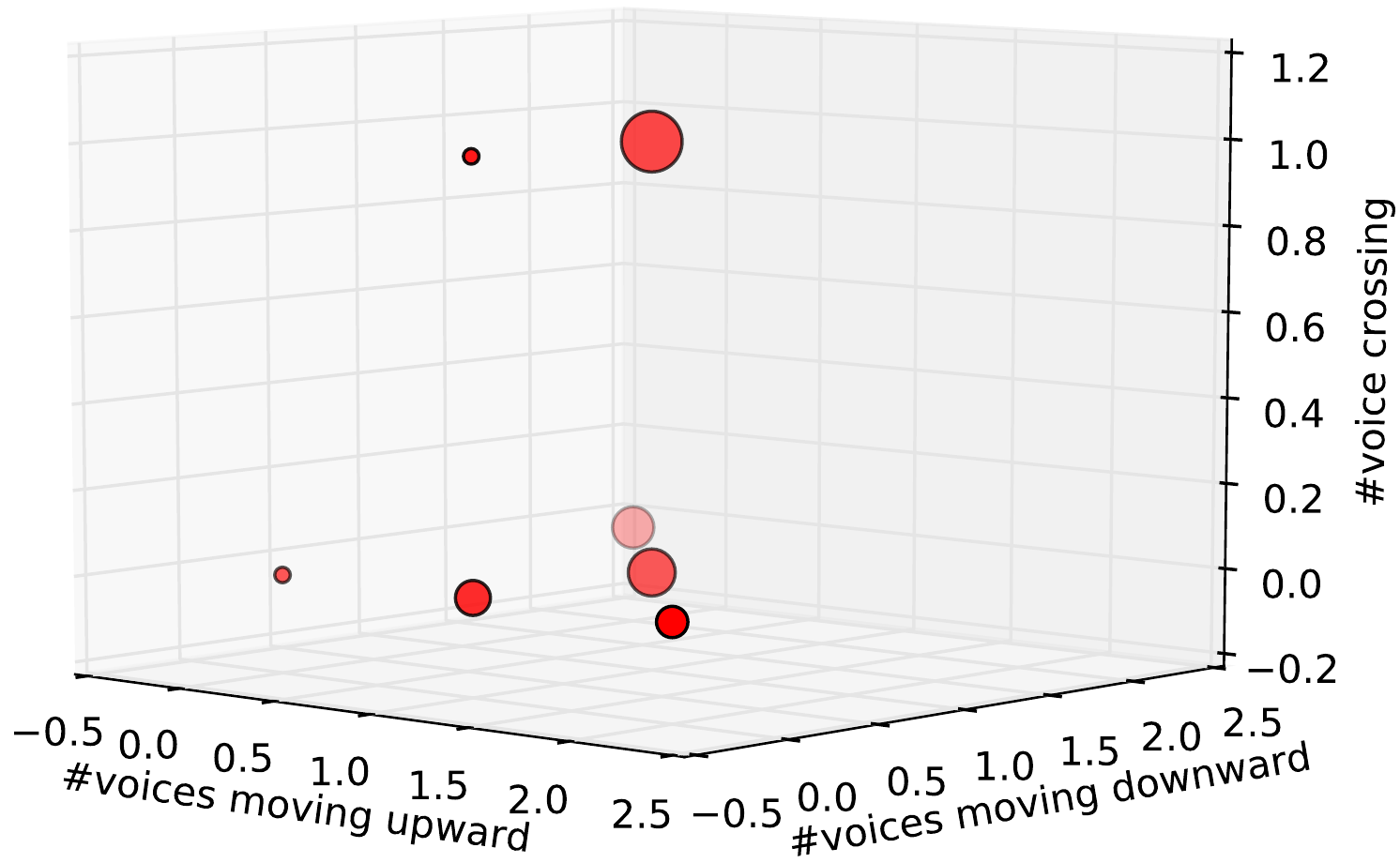}}
\caption{Three-dimensional projections of the complexity cloud of the paradigmatic voice leading \emph{Dicant nunc Judei}.
		The radius of each circle represents the normalised multiplicity of the corresponding complexity vector.} \label{fig:dic}
\end{figure}

\begin{table}[tb]
\caption{Complexity vectors of the analysed fragments and their occurrences.} \label{tab:complexvect}
\centering
\begin{tabular}{cS}
\toprule
\multicolumn{2}{c}{\emph{Angelus Domini}} \\
	$c$		&	{$\mu(c)$} \\
\midrule
$(0, 1, 1, 0)$	&	2 \\
$(0, 1, 1, 1)$	&	2 \\
$(0, 2, 0, 0)$	&	4 \\
$(1, 0, 1, 1)$	&	2 \\
$(1, 1, 0, 0)$	&	6 \\
$(1, 1, 0, 1)$	&	4 \\
$(2, 0, 0, 0)$	&	4 \\
\bottomrule
\end{tabular}
\hspace*{1cm}
\begin{tabular}{cS}
\toprule
\multicolumn{2}{c}{\emph{Dicant nunc Judei}} \\
$c$		& 	{$\mu(c)$} \\
\midrule
$(0, 0, 2, 0)$	&	1 \\
$(0, 2, 0, 0)$	&	7 \\
$(1, 0, 1, 0)$	&	5 \\
$(1, 0, 1, 1)$	&	1 \\
$(1, 1, 0, 0)$	&	9 \\
$(1, 1, 0, 1)$	&	15 \\
$(2, 0, 0, 0)$	&	4 \\
\bottomrule
\end{tabular}
\end{table}

\section{Rhythmic independence and rests}

The examples analysed in Subsection~\ref{subsec:examples} are counterpoints of the first species --- which is the simplest case, in that the voices follow a \emph{note-against-note} flow.
It is however possible to study more complex scenarios by introducing \emph{rhythmic independence} between voices and \emph{rests} in the melody, in any case reducing non-simultaneous voices to the simplest case.

If the voices play at different rhythms or follow rhythmically irregular themes, we consider the minimal rhythmic unit $u$ appearing in the phrase and \emph{homogenise} the composition based on that unit:
if a note has duration $ku$, with $k \in \N$, we represent it as $k$ repeated notes of duration $u$ (see Figure~\ref{fig:reduction} for an example).
This transformation of the original counterpoint introduces only oblique motions and does not alter the number of the other three kinds of motion.

\begin{figure}[tb]
\centering
\subfloat[Counterpoint of the fifth species. \label{fig:rit1}]{\includegraphics[width =\textwidth]{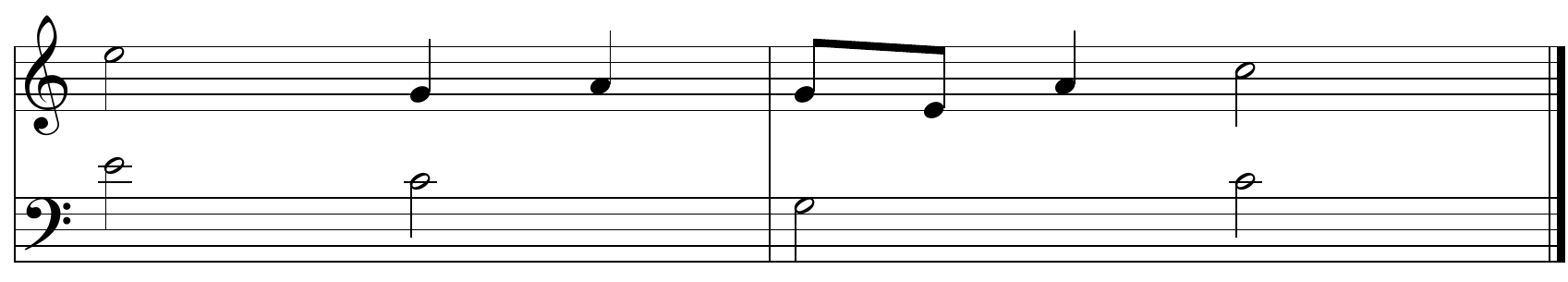}} \\
\subfloat[Reduction to the first species. \label{fig:rit2}]{\includegraphics[width = \textwidth]{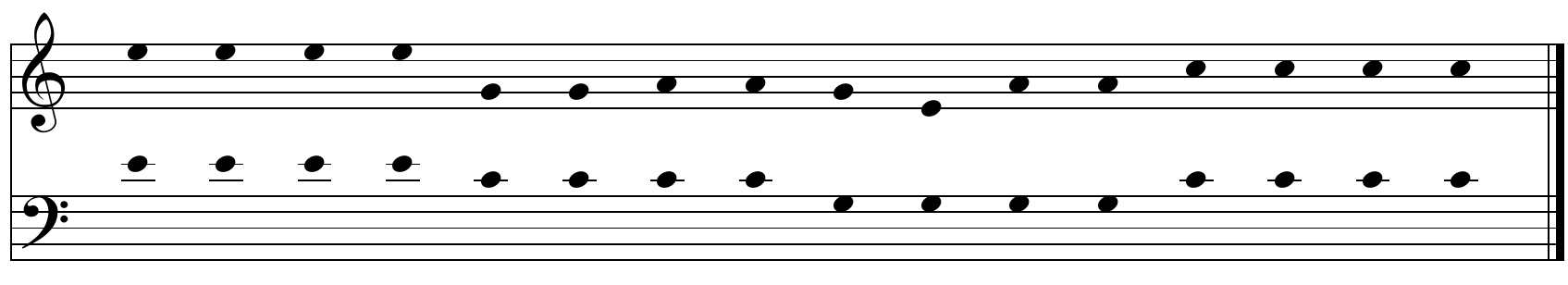}}
\caption{Reduction of rhythmically independent voices to a counterpoint of the first species.}
\label{fig:reduction}
\end{figure}

In musical terms, if a voice is silent it is neither moving nor being constant and it cannot cross other voices. Therefore, in order to include rests in our model it is necessary to slightly modify Algorithm~\ref{alg:ppmat} by
introducing a new symbol ($p$) in the dictionary of pitches; we also choose to indicate a rest in the matrices associated with a voice leading by the entry $-1$.
We adopt the following convention concerning the ordered union multiset.

\begin{conv}
We choose rests to be the last elements in the ordered union multiset associated with a voice leading. In other words, we declare $p$ to be strictly greater than any other pitch symbol.
\end{conv}

\begin{example}
The voice leading $(p , D_4, D_5) \to (D_4 , C_3 , C_3)$ corresponds to the matrix
\[
\begin{pmatrix}
 0 & 0 & 0 & 0 & 0\\
 0 & 0 & 0 & 0 & 0\\
 1 & 0 & 0 & 0 & 0\\
 0 & 1 & 0 & 0 & 0\\
 0 & 0 & -1 & 0 & 0
\end{pmatrix}.
\]
\end{example}


\begin{rem}
Note that when introducing the $-1$'s in the matrix associated with a voice leading we are no longer dealing with partial permutation matrices. However, to study voice leadings with rhythmic independence of the voices as
before (thus ignoring rests) it is enough to consider the minor of the matrix obtained by deleting all rows and columns containing $-1$ (which is obviously again a partial permutation matrix).
\end{rem}

We extend the complexity vector defined previously in Formula~\eqref{eq:complexvect} by adding a fifth component that counts the number of voices that are silent at least once in the voice leading, \ie it counts the number
of negative ($-1$) entries of the associated matrix.
Furthermore, we slightly modify also the notion of \emph{normalised multiplicity} of a complexity vector $c$, needed for the representation of the complexity of a piece in the form of a point cloud, now dividing the number
$\mu(c)$ of occurrences of $c$ in the piece by the total number of notes per voice \emph{after the homogenisation}.

\subsection{Example: the \emph{Retrograde Canon} by J.\,S.~Bach} \label{subsec:canon}

We consider the \emph{Retrograde Canon} (also known as \emph{Crab Canon}), a palindromic canon with two voices belonging to the \emph{Musikalisches Opfer} by J.\,S.~Bach, the beginning of which is reproduced in
Figure~\ref{fig:crabsh}.

\begin{figure}
\centering
\includegraphics[width=\textwidth]{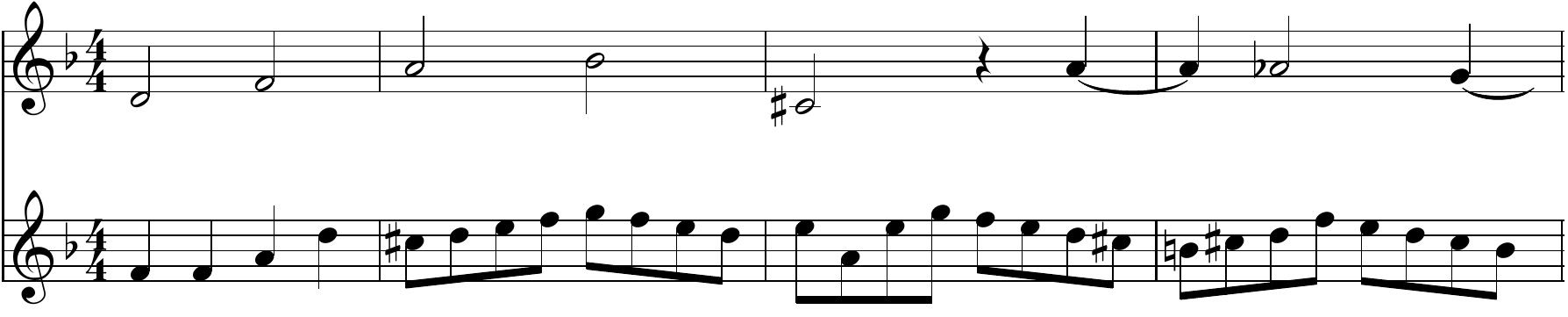}
\caption{The \emph{Retrograde Canon} (bars 1--4), a palindromic canon belonging to the \emph{Musikalisches Opfer} by J.\,S.~Bach.} \label{fig:crabsh}
\end{figure}

We homogenise the rhythm by expressing each note in eighths and we apply Algorithm~\ref{alg:ppmat}. Here is the output of the first four meaningful voice leadings:
\begin{verbatim}
Voice Leading: ['D4', 'D4'] ['D4', 'F4']
c = [1, 0, 1, 0, 0] - oblique motion
Voice Leading: ['D4', 'F4'] ['F4', 'A4']
c = [2, 0, 0, 0, 0] - similar motion up
Voice Leading: ['F4', 'A4'] ['F4', 'D5']
c = [1, 0, 1, 0, 0] - oblique motion
Voice Leading: ['F4', 'D5'] ['A4', 'C#5']
c = [1, 1, 0, 0, 0] - contrary motion
\end{verbatim}
Table~\ref{tab:canon} collects the complexity vectors and their multiplicities; they are displayed in the form of point clouds in Figure~\ref{fig:canon}. 

\begin{table}[tb]
\caption{Complexity vectors of the \emph{Retrograde Canon} and their occurrences.}
\label{tab:canon}
\centering
\begin{tabular}{cScS}
\toprule
\multicolumn{4}{c}{\emph{Retrograde Canon}} \\
	$c$ 		& {$\mu(c)$} &		$c$ 		& {$\mu(c)$} \\
\midrule
$(0, 0, 1, 0, 1)$	&	2 	&	$(1, 0, 0, 0, 1)$	&	2 \\
$(0, 0, 2, 0, 0)$	&	8 	&	$(1, 0, 1, 0, 0)$	&	43 \\
$(0, 1, 0, 0, 1)$	&	2	&	$(1, 0, 1, 1, 0)$	&	1 \\
$(0, 1, 1, 0, 0)$	&	43 	&	$(1, 1, 0, 0, 0)$	&	14 \\
$(0, 1, 1, 1, 0)$	&	1 	&	$(1, 1, 0, 1, 0)$	&	3 \\
$(0, 2, 0, 0, 0)$	&	11 	&	$(2, 0, 0, 0, 0)$	&	11 \\
\bottomrule
\end{tabular}
\end{table}

\begin{figure}[tb]
\centering
\subfloat[Projection on the first three components of the complexity vector. \label{fig:can1}]{\includegraphics[width = 0.45\textwidth]{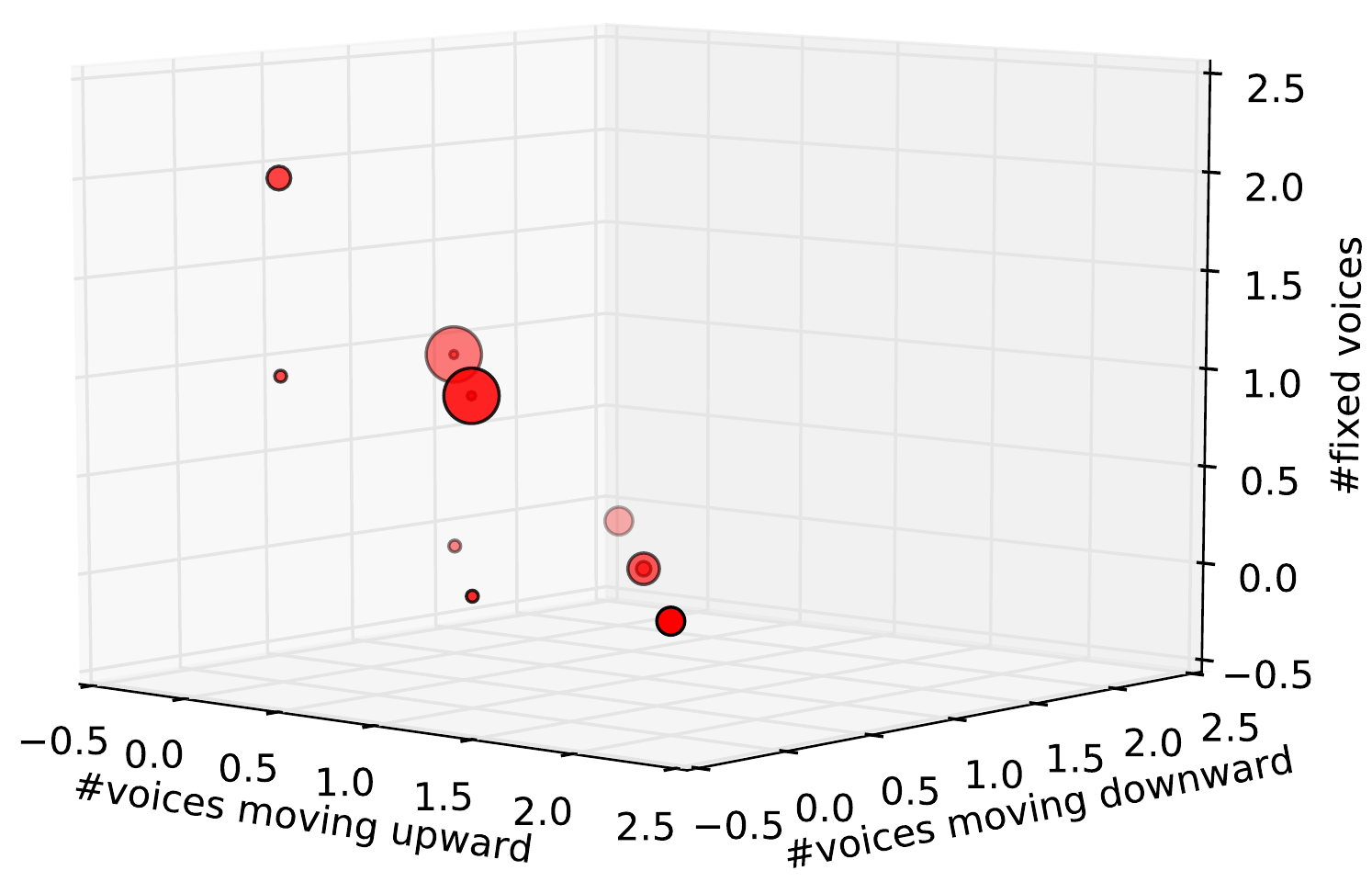}} \qquad
\subfloat[Projection on the \emph{upward, downward} and \emph{crossing} components of $c$. \label{fig:can2}]{\includegraphics[width = 0.45\textwidth]{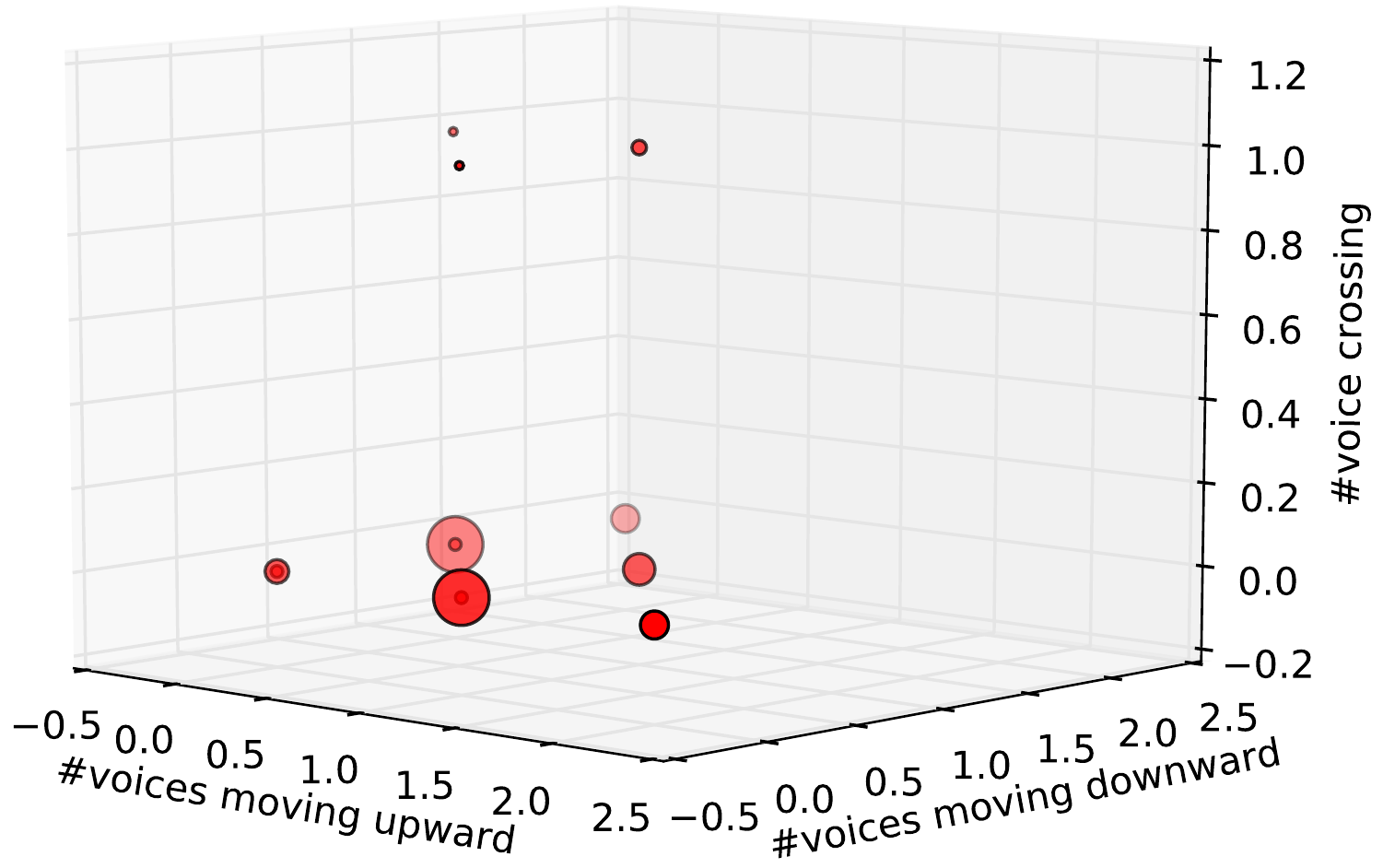}} \\
\subfloat[Projection on the \emph{upward, downward} and \emph{rest} components of $c$. \label{fig:can3}]{\includegraphics[width = 0.45\textwidth]{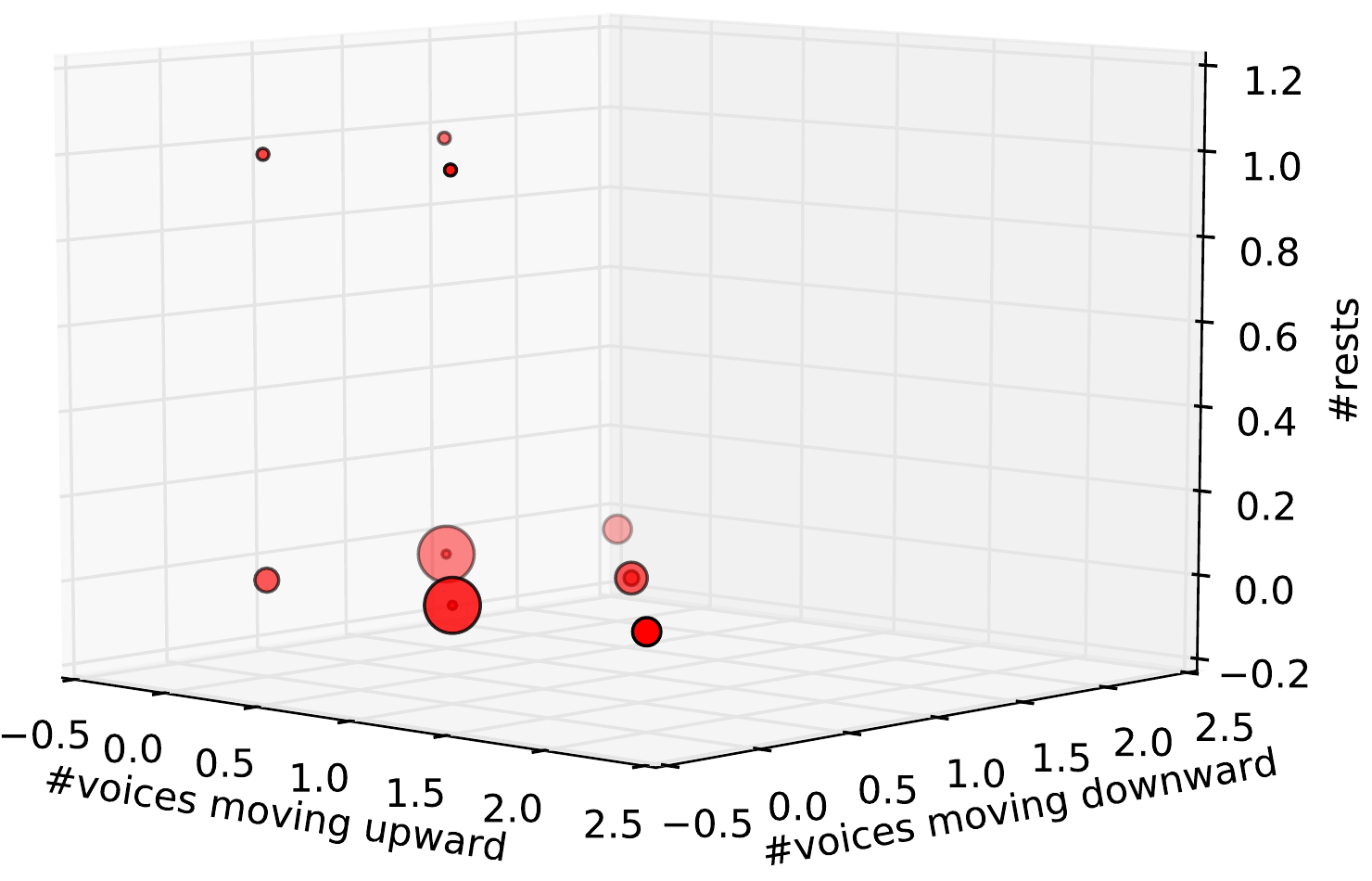}}
\caption{Three-dimensional projection of the $5$-dimensional point cloud representing the complexity of the \emph{Retrograde Canon}.
		The radius of each circle represents the normalised multiplicity of each complexity vector.}
\label{fig:canon}
\end{figure}

\section{Concatenation of voice leadings and time series}

The paradigmatic point cloud associated with a voice leading gives a useful $3$-di\-men\-sio\-nal representation of the piece; however, this analysis is just structural, as it does not take into account the way in which voice
leadings have been concatenated by the composer. It is possible to introduce this temporal dimension by looking at the sequence of complexity vectors from a different viewpoint.

The concatenation of observations in time can be seen as a \emph{time series}, that is a sequence of data concerning observations ordered according to time. 
In our case each piece of music can be described as a $5$-dimensional time series, whose observations are the complexity vectors associated with each voice leading.
More specifically, we use the so-called \emph{dynamic time warping (DTW)}, a method for comparing time-dependent sequences of different lengths: it returns a measure of similarity between two given sequences by
``warping'' them non-linearly (see Figure~\ref{fig:dtw} for an intuitive representation). We invite the reader to consult \cite{senin2008dynamic} for a detailed review of DTW algorithms.

\begin{figure}[tb]
\centering
\includegraphics[width=0.6\textwidth]{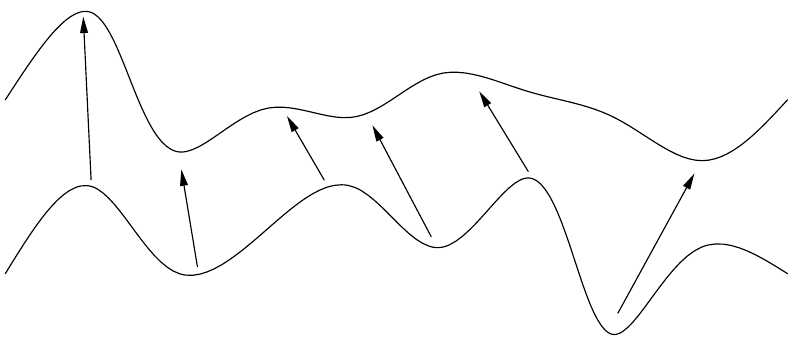} 
\caption{Comparing two time series with DTW.} \label{fig:dtw}
\end{figure}

\subsection {Dynamic time warping analysis}

Let $\mathcal{F}$ be a set, called the \emph{feature space}, and take two finite sequences $X \= (x_1, \dots, x_n)$ and $Y \= (y_1, \dots, y_m)$ of elements of $\mathcal{F}$, called \emph{features} (here $n$ and $m$ are
natural numbers). In order to compare them, we need to introduce a notion of distance between features, that is a map $\mathcal{C}: \mathcal{F} \times \mathcal{F} \to \R$, also called a \emph{cost function}, that
meets at least the following requirements:
\begin{enumerate}[label=\roman*.]
\item $\mathcal{C}(x, y) \geq 0$ for all $x, y \in \mathcal{F}$;
\item $\mathcal{C}(x, y) = 0$ if and only if $x = y$;
\item $\mathcal{C}(x, y) = \mathcal{C}(y, x)$ for all $x, y \in \mathcal{F}$.
\end{enumerate}
Now, if we apply $\mathcal{C}$ to the features $X$ and $Y$, we can arrange the values in an $n \times m$ real matrix $C \= \big(\mathcal{C}(x_i, y_j)\big)$, where $i$ ranges in $\{1, \dotsc, n\}$ and $j$ in $\{1, \dotsc, m\}$.


A \emph{$(n, m)$-warping path} in $C$ is a finite sequence $\gamma \= (\gamma_1, \dotsc, \gamma_l) \in \R^l$, with $l \in \N$, such that:
\begin{enumerate}
\item $\gamma_k \= (\gamma_k^x, \gamma_k^y) \in \{1, \dotsc, n\} \times \{1, \dotsc, m\}$ for all $k \in \{1, \dotsc, l\}$;
\item $\gamma_1 \= (1,1)$ and $\gamma_l \= (n, m)$;
\item $\gamma_k^x \leq \gamma_{k+1}^x$ and $\gamma_k^y \leq \gamma_{k+1}^y$ for all $k \in \{1, \dotsc, l - 1\}$;
\item $\gamma_{k+1} - \gamma_k \in \big\{(1,0), (0, 1), (1,1)\big\}$ for all $k \in \{1, \dotsc, l - 1\}$.
\end{enumerate}
The \emph{total cost} of a $(n, m)$-warping path $\gamma$ over the features $X$ and $Y$ is defined as
\[
	\mathcal{C}_\gamma(X,Y) \= \sum_{k=1}^l \mathcal{C}(x_{\gamma_k^x}, y_{\gamma_k^y}).
\]
An \emph{optimal warping path} on $X$ and $Y$ is a warping path realising the minimum total cost (see Figure~\ref{fig:optimalpath}). We are now ready to define the \emph{DTW distance} between $X$ and $Y$:
\[
	DTW(X,Y) \= \min\Set{\mathcal{C}_\gamma(X,Y) | \gamma \text{ is a $(n, m)$-warping path}}.
\]

\begin{rem}
Note that the minimum always exists because the set is finite.
\end{rem}

\begin{figure}[tb]
\centering
\includegraphics[width=0.5\textwidth]{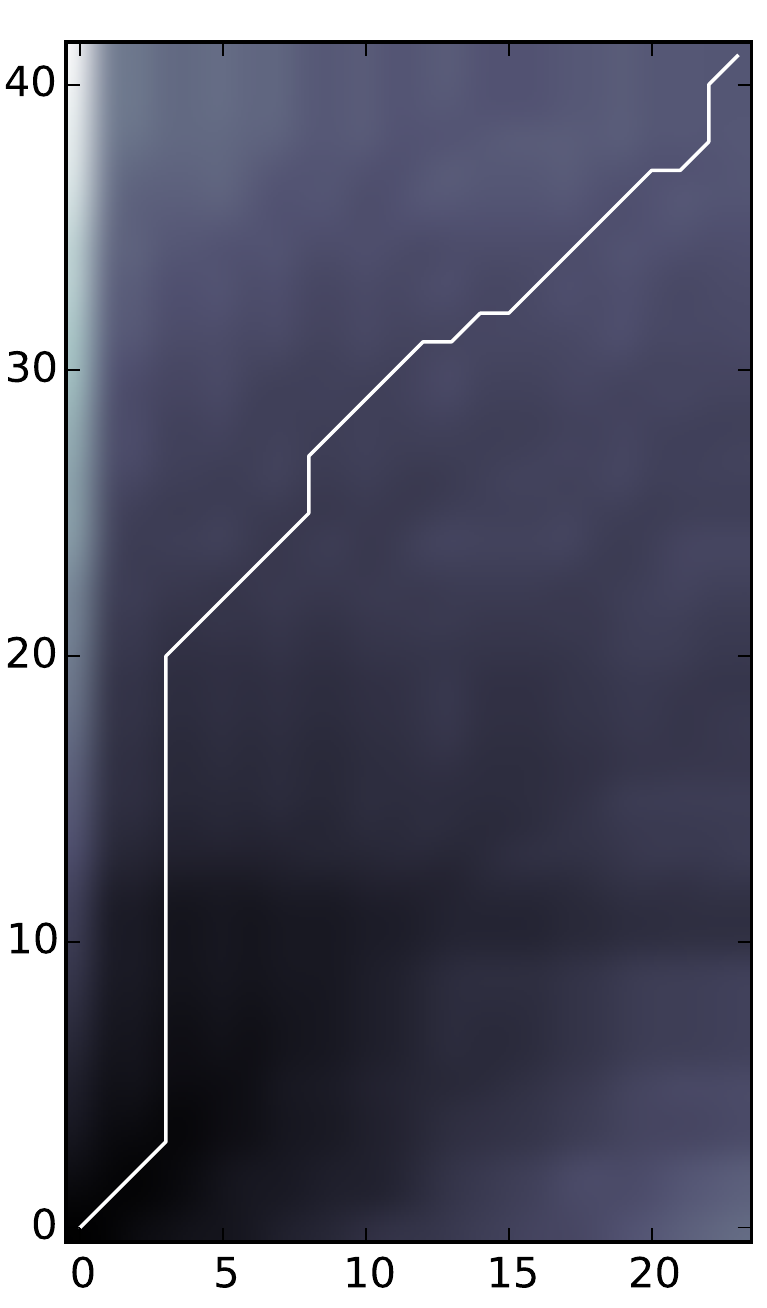} 
\caption{Optimal warping path on \emph{Angelus Domini} and \emph{Dicant nunc Judei}.} \label{fig:optimalpath}
\end{figure}

We computed the DTW distance between each pair of the three examples that we analysed in Subections~\ref{subsec:examples} and \ref{subsec:canon}, choosing as cost function the Euclidean distance in $\R^5$. 
We embedded the $4$-dimensional complexity vectors in $\R^5$ by adding a fifth component and setting it to $0$. The results of the comparison are shown in Table~\ref{tab:dist}. Although we analysed only three
compositions, it is possible to observe how the DTW distance segregates the two pieces belonging to the Chartres fragments.

\begin{table}[tb]
\caption{DTW distance matrix for the three time series of complexity vectors.} \label{tab:dist}
\centering
\begin{tabular}{lccc}
\toprule
			& \emph{Angelus} 	& \emph{Dicant}	& \emph{Canon} \\
\midrule
\emph{Angelus}	&	0.00			& 	0.62			&	1.34 \\
\emph{Dicant}	&	0.62			&	0.00			&	1.16 \\
\emph{Canon}	&	1.34			&	1.16			&	0.00 \\
\bottomrule
\end{tabular}
\end{table}

\section{Conclusion}

Our analysis showed that our definition of complexity in terms of the relative movements of the voices and especially of crossing is suitable for characterising a musical piece. Point-cloud representation yields a ``photograph''
of complexity, a sort of fingerprint that lets clearly emerge what are the main features of the examined composition, noticeable even at first glance. Dynamic time warping provides then further support to this evidence by directly measuring the
distance between the complexities of two pieces, considering each complexity vector as an observation of the piece in time, and giving a quantitative description of the dissimilarity of the time series describing the pieces. 


\bibliographystyle{abbrv}
\bibliography{biblio.bib}


\end{document}